\documentclass[11pt,a4paper]{amsart}
\usepackage{amssymb,amsmath,amscd}
\usepackage[margin=27mm]{geometry}
\usepackage{xcolor,graphicx}
\usepackage{todonotes}
\usepackage{tikz}
\usetikzlibrary{decorations.pathreplacing}

\usepackage[hidelinks]{hyperref}

\input{prepictex}
\input{pictex}
\input{postpictex}

\newcounter{picture}

\newtheorem{theorem}{Theorem}[section]
\newtheorem{lemma}[theorem]{Lemma}
\newtheorem{prop}[theorem]{Proposition}
\newtheorem{coro}[theorem]{Corollary}
\theoremstyle{definition}
\newtheorem{defi}[theorem]{Definition}
\theoremstyle{remark}
\newtheorem{remark}[theorem]{Remark}
\numberwithin{equation}{section}
\newcommand{\cS}{\mathcal S}

\newcommand{\Om}{\Omega}
\newcommand{\w}{\omega}
\newcommand{\G}{\Gamma}

\makeatletter
\def\@setcopyright{}
\def\serieslogo@{}
\makeatother
\begin{document}

\title[Property RD and boundary representations]{Property RD for $\widetilde A_2$ groups\\and boundary representations}
\author{Adrien Boyer}
\address{Universit\'e Paris Cit\'e, IMJ-PRG, 8 place Aur\'elie Nemours, 75013 Paris, France}
\email{adrien.boyer@imj-prg.fr}
\author{Thang Nguyen}
\address{Florida State University, 1017 Academic Way, Tallahassee, FL 32306-4510, USA}
\email{tnguyen9@fsu.edu}
\thanks{T. N. was supported in part by Simons Travel Support for
  Mathematicians grant MPS-TSM-00002547 and NSF grant DMS-2604371.}

\begin{abstract}

We develop a boundary representation approach to Rapid Decay property for groups acting on
buildings of type $\widetilde A_2$.
The central estimate combines a multiplicity argument on the product
of Furstenberg boundary with the Riesz--Thorin theorem. For simply transitive actions on vertices, we also establish the polynomial bound conjectured by Ramagge, Robertson and Steger in the vector-radial case. The cubic degree of this bound is optimal. The method also uses  multiplicity argument with the Riesz--Thorin theorem but on the single Furstenberg boundary.
\end{abstract}
\keywords{Rapid decay, affine buildings, boundary representations, interpolation}
\maketitle

\section{Introduction}
Property {\rm RD} (Rapid Decay) is a fundamental analytic property of discrete groups. Since Haagerup's seminal work on free groups \cite{Ha}, it has played a
fundamental role in geometric group theory, harmonic analysis, ergodic theory,
and operator algebras. Beyond its intrinsic interest, Property {\rm RD} has
found important applications to the Baum--Connes conjecture, reduced group
$C^{*}$-algebras, rigidity phenomena, and quantitative aspects of
representation theory \cite{BoyGa} and \cite{KYD}. One of the first higher-rank examples was obtained by Robertson, Ramagge and
Steger in their influential work \cite{RRS}. They established Property
{\rm RD} for groups acting freely on affine buildings of type
$\widetilde A_{1}\times\widetilde A_{1}$ and $\widetilde A_{2}$.

Let $\Gamma$ be a countable discrete group with a proper length
function $|\cdot|$. A unitary representation
$\pi:\Gamma\to\mathcal U(H)$ satisfies Property {\rm RD} with respect
to this length if there is a polynomial $Q$ such that every finitely supported function \(f:\Gamma\to\mathbb C\) satisfying
$ 
\operatorname{supp}(f)\subseteq
B_r=\{\gamma\in\Gamma:|\gamma|\leq r\},$ 
 we have 
\[
 \|\pi(f)\|_{\mathrm{op}}
 \le Q(r)\|f\|_{\ell^2(\Gamma)}
 \]
where 
$\pi(f)=\sum_{\gamma\in\Gamma}f(\gamma)\pi(\gamma)$.
The group $\Gamma$ has Property {\rm RD} if its left regular
representation $\lambda_\Gamma=\Gamma \to \mathcal{U}(\ell^{2}(\Gamma))$ satisfies this inequality.

In this paper, we give a boundary-representation approach to rapid
decay for groups acting properly and cocompactly on regular affine
buildings of type $\widetilde A_2$. The approach is inspired by the use
of boundary averages in the work of Bader and Muchnik \cite{BM}.

We now fix the geometric setting. Let $X$ be a locally finite,
regular, thick affine building of type $\widetilde A_2$, with
parameter $q\ge2$. Write $X^0$ for its vertex set, $d$ for its graph
distance, and
\[
 \sigma:X^0\times X^0\longrightarrow\mathbb N^2
\]
for its vector distance, also called combinatoiral distance, or shape function. If $\sigma(x,y)=(m,n)$, then
$d(x,y)=m+n$. Let $\Gamma$ act properly and cocompactly on $X$ by
automorphisms preserving $\sigma$. The vertex stabilizers are finite
and uniformly bounded; we write
\[
 K_\Gamma=\sup_{v\in X^0}|\Gamma_v|<\infty.
\]
For a fixed base vertex $x$, set $|\gamma|=d(x,\gamma x)$ and
\[
 S^\Gamma_{m,n}(x)
 =\{\gamma\in\Gamma:\sigma(x,\gamma x)=(m,n)\}.
\]

Let $\Omega$ be the Furstenberg boundary of $X$, viewed as the space
of chambers at infinity, and let $\nu_x$ be its canonical probability
measure based at $x$. The boundary action preserves the measure
class of $\nu_x$ and defines a unitary representation
\[
 \pi_{\nu_x}:\Gamma\longrightarrow
 \mathcal U\bigl(L^2(\Omega,\nu_x)\bigr),
 \qquad
 (\pi_{\nu_x}(\gamma)\phi)(\omega)
 =c_x(\gamma,\omega)^{1/2}\phi(\gamma^{-1}\omega),
\]
where
\[
 c_x(\gamma,\omega)
 =\frac{d(\gamma_*\nu_x)}{d\nu_x}(\omega).
\]

The passage from this representation to the regular representation
uses Shalom's spectral comparison lemma \cite[Lemma~2.3]{Sh}.
The constant function $\mathbf1_\Omega$ is a nonzero positive vector
for $\pi_{\nu_x}$, so
\[
 \|\lambda_\Gamma(\mu)\|_{\mathrm{op}}
 \le\|\pi_{\nu_x}(\mu)\|_{\mathrm{op}},
\]
where $\mu$ is positive and bounded.
Consequently, for every finitely supported complex-valued $f$,
\[
 \|\lambda_\Gamma(f)\|_{\mathrm{op}}
 \le\|\lambda_\Gamma(|f|)\|_{\mathrm{op}}
 \le\|\pi_{\nu_x}(|f|)\|_{\mathrm{op}}.
\]
It therefore suffices to establish bounds for the
boundary representation.

Write
\[
 P_{\mathrm{RRS}}(m,n)
 =\frac12(m+1)(n+1)(m+n+2).
\]
For a free action on vertices, Ramagge, Robertson and Steger proved
the estimate
\[
 \|\lambda_\Gamma(f)\|_{\mathrm{op}}
 \le P_{\mathrm{RRS}}(m,n)
       \sqrt{\max\{m,n\}+1}\,\|f\|_{\ell^2(\Gamma)}
\]
when $f$ is supported on $S^\Gamma_{m,n}(x)$, and conjectured that
the factor $\sqrt{\max\{m,n\}+1}$ could be omitted \cite{RRS}.
The general estimate obtained by the boundary argument is as follows.

\begin{theorem}\label{maintheorem}
Assume that $\Gamma$ acts cocompactly by isometries, preserving $\sigma$ on $X$ as above. For every $x\in X^0$, every
$(m,n)\in\mathbb N^2$, and every function $f$ supported on
$S^\Gamma_{m,n}(x)$,
\[
 \|\pi_{\nu_x}(f)\|_{\mathrm{op}}
 \le\sqrt{K_\Gamma}\,
 P_{\mathrm{RRS}}(m,n)\sqrt{m+n+1}\,
 \|f\|_{\ell^2(\Gamma)}.
\]
In particular, $\Gamma$ has Property {\rm RD} with respect to
$|\gamma|=d(x,\gamma x)$.
\end{theorem}

Valette \cite{Val} established a Haagerup inequality for functions radial
with respect to graph distance on groups acting simply transitively on
the vertices of thick affine buildings of type $\widetilde A_n$.

For functions constant on a vector-distance sphere, uniform counting
on the boundary gives a sharper estimate in the simply transitive
case. We call a function supported on $S^\Gamma_{m,n}(x)$
\emph{vector-radial} if it is of the form
$f=c\mathbf1_{S^\Gamma_{m,n}(x)}$.

\begin{theorem}\label{theorem2}
Assume in addition that $\Gamma$ acts simply transitively on $X^0$.
For every $(m,n)\in\mathbb N^2$ and every vector-radial function $f$
supported on $S^\Gamma_{m,n}(x)$,
\[
 \|\pi_{\nu_x}(f)\|_{\mathrm{op}}
 \le P_{\mathrm{RRS}}(m,n)\|f\|_{\ell^2(\Gamma)}.
\]
\end{theorem}

This gives a boundary-representation proof of the cubic radial
behaviour discussed in \cite{RRS}. The degree is optimal for
spherewise estimates in this setting. To see the lower bound, let
\[
 \Xi_x(\gamma)
 =\langle\pi_{\nu_x}(\gamma)\mathbf1_\Omega,
                  \mathbf1_\Omega\rangle.
\]
The $\widetilde A_2$ specialization of the spherical-function
formula in \cite{Par} gives
\[
 \Xi_x(\gamma)
 \asymp(m+1)(n+1)(m+n+1)q^{-(m+n)}
 \quad\text{if }\sigma(x,\gamma x)=(m,n).
\]
Amenability of the boundary action \cite{Lec} implies
$\pi_{\nu_x}\prec\lambda_\Gamma$. Combining this with Shalom's
lemma gives equality of the two operator norms for nonnegative
coefficients. Applying this equality to
$f=\mathbf1_{S^\Gamma_{m,n}(x)}$ and evaluating on
$\mathbf1_\Omega$ yields
\[
 \frac{\|\lambda_\Gamma(f)\|_{\mathrm{op}}}{\|f\|_2}
 \ge\Xi_x(m,n)\sqrt{|S^\Gamma_{m,n}(x)|}
 \asymp(m+1)(n+1)(m+n+1),
\]
where $\Xi_x(m,n)$ denotes the common value on the sphere.

The proof separates foldings into two types. Type~A foldings are
controlled by a multiplicity estimate on rectangles in the boundary
product. For type~B foldings, we first compress the boundary
operators using the two folding vertices. On the product boundary,
the diamond geometry controls both the direct and inverse
multiplicities of each block. The corresponding Radon--Nikodym
weights are reciprocal, so Riesz--Thorin interpolation gives an
$L^2$ bound independent of the lengths of the outer pieces.
Localization of pure tensors then allows summation over the
branching vertices, followed by summation over the branching
levels. For vector-radial functions, a direct counting argument on
the boundary yields the stronger bound of Theorem~\ref{theorem2}.

Section~\ref{Sec2} introduces the building, its Furstenberg boundary, its boundary measures,
and the Busemann cocycle. Section~\ref{foldingssec} describes the
foldings and their boundary partitions, and
Section~\ref{Sec4} defines the associated compressed operators.
Section~5 proves the product-boundary estimates for type~B
foldings. Section~6 treats type~A foldings and combines the
estimates to prove Theorem~\ref{maintheorem}. Finally,
Section~\ref{sec:radial-boundary} proves the vector-radial estimate.

\subsection{Acknowledgments}

The authors are grateful to the Weizmann Institute of Sciences, Courant Institute of Mathematical Sciences, the University of Michigan, the Institut des
Hautes \'Etudes Scientifiques, the Institut Henri Poincar\'e, LabEx CARMIN
(ANR-10-LABX-59-01), and the Vietnam Institute for Advanced Study in
Mathematics for their support and hospitality.

The first author warmly thanks Odylo Costa, Thomas Haettel, Jean Lécureux, Christophe Pittet, and Rafel Potrie for helpful discussions,
and is deeply grateful to IRL 2030 IFUMI (Institut Franco-Uruguayen des
Mathématiques et Interactions), whose support made this work possible.

%Eventually, the authors thank the Weizmann Institute of Sciences where this project began.
\subsection{AI use}

AI was used to check for errors and assist with grammar and polishing of the writing in the final stages of the preparation of this manuscript; none of our mathematical arguments were produced with AI assistance.

\section{Geometric preliminaries}\label{Sec2}
\subsection{Definition of $\tilde{A_{2}}$ buildings}

We start with a model apartment (Coxeter complex) $\Sigma_{0}$ that is a 2-simplicial complex viewed as the Euclidean plane $\mathbb{R}^{2}$ with a tessellation of equilateral triangles. The collection $\mathcal{H}$ of lines that are affine hyperplanes denoted by $H$, are called the walls in $\Sigma_{0}$. The set $\Sigma \backslash \cup_{H \in \mathcal{H}} H$ is made of open convex equilateral triangles called chambers in $\Sigma_{0}$.

\begin{defi}
A building of type $\tilde{A}_{2}$ is a simplicial complex which contains a family of subcomplexes called apartments such that:
\begin{enumerate}
\item Each apartment is isomorphic (as simplicial complex) to $\Sigma_{0}$.
\item Any two simplices are contained in an apartment.
\item Given two apartments $\Sigma,\Sigma'$, there is a simplicial isomorphism $\Sigma \to \Sigma'$ that fixes each vertex in $\Sigma\cap\Sigma'$. 
\end{enumerate}
\end{defi}

An $\tilde{A_{2}}$-building will be denoted by $X$. The chambers of $X$ are the equilateral triangles of $2$-dimensional simplices. Let $\Sigma$ be an apartment. A wall of $\Sigma$ in $X$ is a (Euclidean) line formed by edges of the tessellation.  If we fix an origin $x\in \Sigma$, the set of walls passing through $x$ divide $\Sigma$ into $6$
connected components called \emph{sectors} (or Weyl chambers). Moreover, $X$ is a thick building: given an edge, the number of adjacent chambers, namely chambers having a common edge,  is greater than $3$ and regular: this number is fixed given by $q+1$ with $q\in \mathbb{N}$ and $q\geq 2.$ \\ 
%Given any apartment $\Sigma$ and any chamber $C\in \Sigma$ there exists a canonical map called the retraction onto $\Sigma$ centered at $C$ denoted by $\rho_{\Sigma,C}:X \to \Sigma $. We refer to \cite[Chapter IV, 3) Retractions, p85]{Bro}. \\
An $\tilde{A_{2}}$-building is equipped with a coloring of its vertices set $X^{0}$ by three colors $\{0, 1, 2 \}$. Each chamber has exactly one vertex of each color. The color of a vertex is called its type and thus we have a map $\tau :X^{0} \to \mathbb{Z}/3\mathbb{Z}.$ We say that $\phi:X \to X$ is a type rotating automorphism if $\phi$ is a bijection sending simplices to simplices and preserving the type: there exists $t\in \{0,1,2\} $ such that $\tau(\phi(x))=\tau(x)+t$ for any vertex $x\in X^{0}.$

\subsection{Combinatorial distance, spheres and convex hull}\label{combin}
A building of type $\tilde{A_{2}}$ is an example of $CAT(0)$-space when it is endowed with a natural metric, defined as follows: the distance between two points $x$ and
$y$ is the distance between $x$ and $y$ calculated in any apartment containing $x$ and $y$. Nevertheless, there is another object ``measuring" the distance between two points that will play an important role in the following: the combinatorial distance. We fix $\Lambda$ be a linear simplicial cone in $\Sigma_{0}$ with open angle of $\pi/3$ based at a vertex $o$ called a fundamental sector. We consider $\Lambda$ as tessellated by equilateral triangles on the sides $1$. The vertex set of $\Lambda$ are parametrized by $\mathbb{N}^{2}.$ Moreover, the type of $(i,j)\in \mathbb{N}^{2}$ is defined to be
$i+2j+3\mathbb{Z} \in \mathbb{Z} / 3\mathbb{Z}$. Thus, for $t \in \{0,1,2 \}$, the set
$ \{(i, j) \in  \Lambda | i \equiv j +t\pmod 3\}$
denotes the set of vertices of type $t$. A sector in $X$ is a subset $\mathcal{S}$ of some apartment $\Sigma$ which is isomorphic to a simplicial cone. Namely, there exists a type rotating isomorphism $\phi:\Sigma\mapsto \Sigma_{0}$ sending $\mathcal{S}$ to $\Lambda$. If $x$ and $y$ are in some sector $\mathcal{S}$ and $\phi_{x}:\mathcal{S} \to \Lambda$ with $\phi_{x}(x)=(0,0)$ and $\phi_{x}(y)=(m,n)$ then the vector $(m,n)\in \Lambda$ does not depend on a particular apartment containing $\mathcal{S}$ it is called the combinatorial distance, denoted by $\sigma(x,y).$ We define the norm of ${\bf{n}}=(m,n)\in \Lambda$ as $|{\bf{n}}|=m+n$, taking the values in $\mathbb{N}$ that coincides with the graph metric $d$ where 
\begin{equation}\label{lengthfunction}
d(x,y)=m+n.
\end{equation} It is worth noting that $\sigma : X^{0}\times X^{0}\to \mathbb{N}^{2}$ is not at all symmetric and if $y$ belongs to the combinatorial sphere $S_{m,n}(x)=\{y\in X^{0} |\sigma(x,y)=(m,n) \}$ , then $x\in S_{n,m}(y).$    An automorphism of $X$ is type-rotating if and only if it
preserves vector distance; see \cite[Section~1.2]{RRS}.
Thus, for every type-rotating automorphism $\gamma$,
\[
\sigma(\gamma x,\gamma y)=\sigma(x,y)
\qquad(x,y\in X^0).
\]
 
 %Associated to $(k_{1},k_{2})\in \Lambda$ define, when the intersection $S_{k_{1},k_{2}}(x)\cap CH(x,y)$ is not empty, the unique point \begin{equation}
%y_{k_{1},k_{2}}=S_{k_{1},k_{2}}(x)\cap CH(x,y)
%\end{equation}
 %To simplify terminology, we call $\textbf{n}=(n_{1},\dots,n_{d})$ the $x$-based coordinate of $y$. Note also that if $y\in S_{\textbf{n}}(x)$ then $x\in S_{\textbf{n}^{\tau}}(y)$ where $\tau$ flips the coordinates.
 The combinatorial convex hull of two vertices $x,y\in \Sigma$,
denoted by $CH(x,y)$, is the minimal convex subcomplex of $\Sigma$ that contains $x$ and $y$. It contains the metric convex hull. Equivalently, this is the intersection of all the apartments containing $x,y$. It is either a parallelogram given by his sides of length $m$ and $n$ if $m,n\neq 0$, or an edge of length $m$ (if $n=0$) or $n$ (if $m=0$). In this last case, we denote the convex hull of $x,y$ by $[x,y].$ 

%A sector based at some vertex $X^{0}$ contains a unique tetrahedron $\lbrace X^{0},\dots,x_{n}\rbrace$, also called a base chamber $c$ based at $X^{0}$. Then $c$ defines $n$ walls: $[X^{0},x_{1}]$, $[X^{0},x_{2}]$, $\dots$, $[X^{0},x_{d}]$. We have $\tau(x_{i})=i$.
%The set $\Om$ is in fact the set of chambers of the spherical building $X^{\infty}$ at
%infinity associated to $X$. 
%The boundary $\Om$ can be expressed as a union of $(q^2+q+1)(q+1)$ disjoint sets
%since, for any fixed vertex $v$,
%\[
%\Om = \bigcup_{ \text{chambers } C  \text{with } v\in C}\Om_v^C
%\]
%where $\Om_v^C$ denotes the set of $\w\in\Om$ whose representative sector
%based at $v$ has base chamber $C$.
%For brevity we will denote the region $I_{e}^u(\w)$ simply by $I^u(\w)$.

\subsection{The boundary and its measures}
Two sectors are parallel if they contain a common subsector.
The Furstenberg boundary $\Omega$ is the set of parallelism classes. Every
$\omega\in\Omega$ has a unique representative sector
$\mathcal S_x(\omega)$ based at any prescribed vertex $x$.
For $y\in X^0$, define the shadow
\[
 \Omega_x(y)=\{\omega\in\Omega:y\in\mathcal S_x(\omega)\}.
\]
These shadows form a basis of clopen sets for the topology on
$\Omega$, which is a totally disconnected compact Hausdorff space.
This topology is independent of the vertex $x$~\cite[Lemma~2.5]{CMS}.

Given a sector $\mathcal{S}(\w)\subset X$, left multiplication of every vertex in
$\cS(\w)$ by an element $\gamma\in\Gamma$ defines an action of $\G$ on $\Om$. We
refer the
reader to~\cite{CMS} for a proof that this action is in fact well-defined.

For each $\mathbf r\in\mathbb N^2$,
\begin{equation}\label{partition1}
 \Omega=\coprod_{y\in S_{\mathbf r}(x)}\Omega_x(y).
\end{equation}
The cardinality $N_{\mathbf r}=|S_{\mathbf r}(x)|$ is independent of
$x$. The standard boundary probability measure $\nu_x$ is characterized by
\begin{equation}\label{slemma}
 \nu_x(\Omega_x(y))=N_{\mathbf r}^{-1},
 \qquad y\in S_{\mathbf r}(x).
\end{equation}
We use the standard boundary construction of \cite{CMS,CR}.
The measures $\nu_x$ are mutually absolutely continuous and equivariant:
$\gamma_*\nu_x=\nu_{\gamma x}$ for every type-rotating automorphism $\gamma$.

%\subsection{The group }
%Let $\Gamma$ act properly and cocompactly on $X$ by automorphisms
%preserving $\sigma$. Fix $x\in X^0$ and set
%\begin{equation}\label{constantK}
% K_\Gamma=\sup_{v\in X^0}|\Gamma_v|<\infty.
%\end{equation}
%Indeed, the stabilizers are finite by properness and there are finitely
%many vertex orbits. A nonempty transporter
%$\{\gamma:\gamma v=w\}$ is a coset of $\Gamma_v$, and hence has
%cardinality at most $K_\Gamma$.
%Define
%\begin{equation}\label{lengthfunction}
 %|\gamma|=d(x,\gamma x),\qquad
 %S^\Gamma_{m,n}(x)=\{\gamma:\sigma(x,\gamma x)=(m,n)\}.
%\end{equation}

\subsection{Branching along singular directions}\label{branching}
For $\mathbf e\in\{(1,0),(0,1)\}$ and $u\in S_{N\mathbf e}(b)$,
the singular segment $[b,u]$ is unique. Indeed, it is a straight
Euclidean geodesic in an apartment, and geodesics are unique in the
CAT(0) realization of $X$; see \cite{Bro}.
Two such segments based at $b$ cannot separate and then meet again.
Consequently, for $u\in S_{N\mathbf e}(b)$ and
$u'\in S_{N'\mathbf e}(b)$, there is a unique vertex $s$ such that
\begin{equation}\label{gromovprod}
 [b,u]\cap[b,u']=[b,s].
\end{equation}
We write $(u,u')_b=d(b,s)$ for the common initial length. Thus
$0\le(u,u')_b\le\min(N,N')$. This notation records branching
along singular segments, rather than a Gromov product for the full
metric space.

\subsection{The Busemann cocycle}
For $x,y\in X^0$ and $\omega\in\Omega$, choose a vertex $z$ sufficiently
far inside a common subsector of $\mathcal S_x(\omega)$ and
$\mathcal S_y(\omega)$. Define
\[
 \beta_\omega(x,y)=\sigma(x,z)-\sigma(y,z)\in\mathbb Z^2.
\]
This does not depend on $z$, and it satisfies the cocycle identity.
The boundary measures satisfy \cite[Lemma~1.1.6]{CR}
\begin{equation}\label{RN}
 \frac{d\nu_y}{d\nu_x}(\omega)
 =q^{\langle\beta_\omega(x,y),\mathbf h\rangle},
\end{equation}
where $\mathbf h=(2,2).$

\section{Foldings}\label{foldingssec}
Fix $\omega\in\Omega$ and an apartment $\Sigma$ containing a sector
representing $\omega$. The retraction
$\rho_{\omega,\Sigma}:X\to\Sigma$ is defined using apartments containing
a common subsector with $\Sigma$. It restricts to an isometry on each
such apartment and fixes $\Sigma$ pointwise.

The retraction of two adjacent chambers either leaves them adjacent
or maps them onto the same chamber. In the latter case, their common
panel is a folding edge. Applied to $CH(x,y)$, this gives the folding
configuration of \cite[Section 1.5]{RRS}. The boundary point determines
the direction marked by an arrow in the figures below.

For comparison, in a tree the segment from $x$ to $y$ folds at a single
vertex $z$, with $d(x,z)+d(z,y)=d(x,y)$; see Figure~\ref{folding tree}.
In type $\widetilde A_2$, the configuration is recorded by two folding
vertices $z_1,z_2\in CH(x,y)$, ordered from the $x$-side to the $y$-side.
Their connecting singular segment may degenerate to a point.
Writing $\sigma(x,y)=\mathbf n=(m,n)$, set
\[
\mathbf i=\sigma(x,z_1),\qquad
\mathbf j=\sigma(y,z_2).
\]

This construction depends on $\omega$, not on the choice of the target
apartment of the retraction. We write
\[
\mathcal F(x,y,\omega)=(\mathbf i,\mathbf j) \in \mathbb{N}^{2}\times \mathbb{N}^{2},
\]
and denote the set of possible indices by $\mathcal F(\mathbf n)$.

\begin{defi}\label{goodandbadfold}
A folding index is of type~A if
$|\mathbf i|+|\mathbf j|=|\mathbf n|$, and of type~B if
$|\mathbf i|+|\mathbf j|<|\mathbf n|$.
The corresponding index sets are denoted by
$\mathcal F_A(\mathbf n)$ and $\mathcal F_B(\mathbf n)$.
\end{defi}
\begin{remark}
    The indexing convention for a nondegenerate type~A folding is fixed
by $i_1\ge m-j_2$, as illustrated in Figure~\ref{zs in A2}.
\end{remark}
Figures~\ref{folding A2 walls}, \ref{BfoldingA2}, and \ref{GfoldingA2}
show the singular and parallelogram configurations. The estimates below
concern the type~B foldings. On such a piece, the two outer sectors pass
through the folding vertices, while the remaining configuration between
$z_1,z_2$ is an extreme folding.

\refstepcounter{picture}
\begin{figure}[htbp]
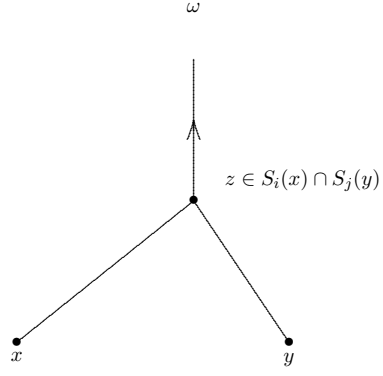
\label{folding tree}
{}\hfill
\scalebox{0.72}{\beginpicture
\setcoordinatesystem units <0.5cm,0.866cm>
\setplotarea x from -6 to 6, y from -1 to 6
\put {$\bullet$} at -5 0
\put {$\bullet$} at 5 0
\put {$\bullet$} at 1.5 3
\put {$x$}                 [t]  at    -5  -0.2
\put {$y$}                 [b]  at    5  -0.5
\put {$\w$}                 [b]  at    1.5  7
\put{$z\in S_{i}(x)\cap S_{j}(y)$} [b]  at    5.5 3.2

\setlinear
\plot -5  0   1.5 3   5 0 /
\plot 1.5 3   1.5 6 /

\arrow <10pt> [.3,.67] from 1.5 4.5 to 1.5 4.7

\endpicture}
\hfill{}
\caption{A folding in the case of a tree.}
\end{figure}

\refstepcounter{picture}
\begin{figure}[htbp]
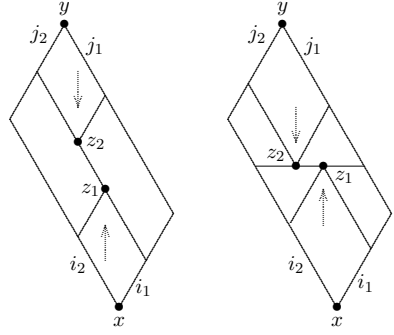
\label{zs in A2}
\hfil
\scalebox{0.72}{\beginpicture
\setcoordinatesystem units <0.5cm,0.866cm>
\setplotarea x from -3 to 11, y from -2.5 to 4.5

\setlinear
\plot 1  -2    3  0   -1  4     -3 2   1 -2   /
\plot 2 -1   -2 3 /
\plot -0.5 -0.5   0.5 0.5   /
\plot -0.5  1.5    0.5 2.5   /
\put{$\bullet$} at  1    -2
\put{$x$} [t]   at  1    -2.2
\put{$\bullet$} at -1     4
\put{$y$} [b]   at -1     4.2
\put{$\bullet$}  at  -0.5   1.5
\put{${z_{2}}$} [l] at  -0.2   1.5
\put{$\bullet$} at  0.5 0.5
\put{${z_{1}}$} [r] at  0.3   0.5
\put{$i_1$} [tl] at  1.6   -1.4
\put{$i_2$} [tr] at  -0.2   -0.9
\put{$j_2$} [br] at  -1.6   3.6
\put{$j_1$} [bl] at  -0.2   3.35
\setdots <1pt>
\arrow <6pt> [.3,.67] from    0.5  -1     to    0.5   -0.25
\arrow <6pt> [.3,.67] from   -0.5   3     to   -0.5    2.25

\setsolid
\putrule from  6 1 to  10  1
\plot 9  -2    11  0   7  4     5 2  9 -2   /
\plot  7.5   1      5.75 2.75 /
\plot 10.25 -0.75   8.5  1 /
\plot  8.5   1      7.3 -0.2  /
\plot  7.5   1      8.7 2.27   /
\put{$\bullet$}  at  9    -2
\put{$x$} [t]    at  9    -2.2
\put{$\bullet$}  at  7     4
\put{$y$} [b]    at  7     4.2
\put{$\bullet$}  at  7.5   1
\put{${z_{2}}$} [br] at  7.1   1.1
\put{$\bullet$}  at  8.5   1
\put{${z_{1}}$} [tl] at  8.92   0.9
\put{$i_{1}$} [tl] at  9.75  -1.3
\put{$i_{2}$} [tr] at  7.8   -0.9
\put{$j_2$} [br] at  6.4   3.6
\put{$j_1$} [bl] at  7.8   3.35
\setdots <1pt>
\arrow <6pt> [.3,.67] from    8.5    -0.25  to    8.5      0.5
\arrow <6pt> [.3,.67] from    7.5     2.25  to    7.5      1.5
\endpicture}
\hfil
\caption{Positions of vertices $z_{1}$ and $z_{2}$.}
\end{figure}

\refstepcounter{picture}
\begin{figure}[htbp]
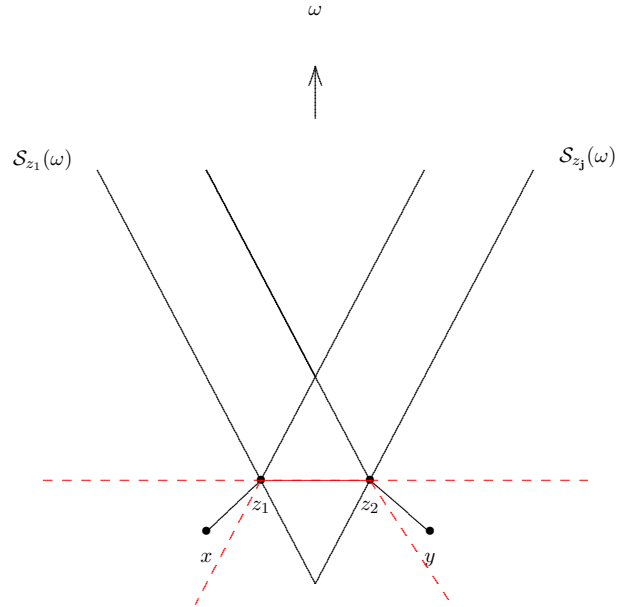
\label{folding A2 walls}
{}\hfill
\scalebox{0.72}{\beginpicture
\setcoordinatesystem units <1cm,1.9cm>
\setplotarea x from -6 to 6, y from -2 to 6
\put {$\bullet$} at 0 0

\put {$\bullet$} at 2 0

\put {$\bullet$}                 [t]  at    -1  -0.45
\put {$\bullet$}                 [t]  at    3.1  -0.45

\put {$z_{1}$}                 [t]  at    0  -0.2
\put {$z_{2}$}                 [t]  at    2  -0.2
\put {$\w$}                 [b]  at    1  4.5
\put{$\cS_{z_{1}}(\w)$}  [b]  at    -4  3
\put{$\cS_{z_{\bf{j}}}(\w)$}  [b]  at    6  3

\put {$x$}                 [t]  at    -1  -0.7
\put {$y$}                 [t]  at    3.1  -0.7

\arrow <10pt> [.3,.67] from 1 3.5 to 1 4
\setlinear
\plot -3  3   0 0   3 3 /

\plot  1  1  -1 3    /
\plot -1  3   2 0   5 3 /
\plot -1  -0.5   0 0   /
\plot  2 0   3.1 -0.5 /
\plot 0 0   1 -1  2 0 /

\color{red}
\plot 0 0 2 0 /

\setdashes
\putrule from -4    0     to  6    0
\plot  0 0  -1.2 -1.2 /
\plot  2 0  3.5 -1.2 /

\endpicture}
\hfill{}
\color{black}
\caption{A type B folding when $CH(x,y)$ is a segment.}
\end{figure}

\refstepcounter{picture}
\begin{figure}[htbp]
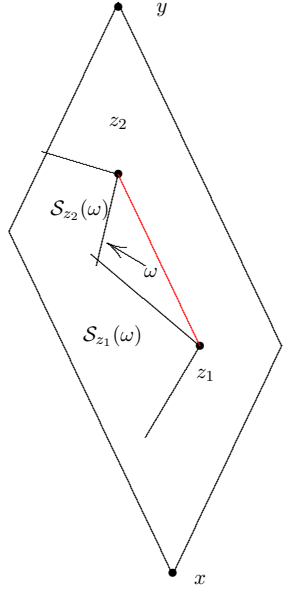
\label{BfoldingA2}
{}\hfill
\scalebox{0.72}{\beginpicture
\setcoordinatesystem units <1cm,2.1cm>
\setplotarea x from -6 to 6, y from -2 to 6
\put {$\bullet$} at 0 2.5
\put {$\bullet$} at 1.5 1

\put {$\bullet$}                 [t]  at    1  -0.95
\put {$\bullet$}                 [t]  at    0  4

\put {$z_{2}$}                 [t]  at    0  3
\put {$z_{1}$}                 [t]  at    1.6  0.8
\put {$\w$}                 [b]  at    0.6  1.6
\put{$\cS_{z_{1}}(\w)$}  [b]  at    -0.1  1
\put{$\cS_{z_2 }(\w)$}  [b]  at   -0.7  2.1

\put {$x$}                 [t]  at    1.5  -1
\put {$y$}                 [t]  at    0.8  4

\arrow <10pt> [.2,.6] from 0.5 1.7 to -0.2 1.9
\setlinear
\plot    1 -1 -2  2  0 4   3 1 1 -1 /
\plot    0 2.5 -0.4 1.7 /
\plot 0 2.5 -1.4 2.7 /
\plot -0.5 1.8 1.5 1 /
\plot 1.5 1 0.5 0.2 /

\color{red}
\plot 0 2.5 1.5 1 /

\setdashes

\endpicture}
\hfill{}
\color{black}
\caption{A type B folding when $CH(x,y)$ is a parrallelogram. The retraction is centered at $\omega$.}
\end{figure}

\refstepcounter{picture}
\begin{figure}[htbp]
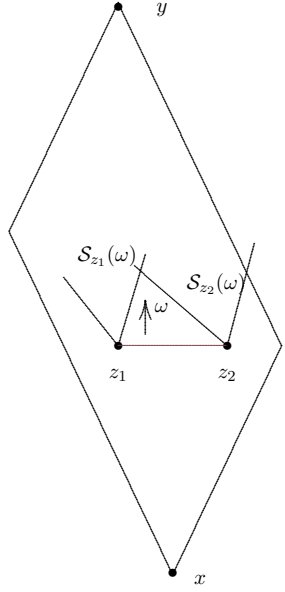
\label{GfoldingA2}
{}\hfill
\scalebox{0.72}{\beginpicture
\setcoordinatesystem units <1cm,2.1cm>
\setplotarea x from -6 to 6, y from -2 to 6
\put {$\bullet$} at 0 1

\put {$\bullet$} at 2 1

\put {$\bullet$}                 [t]  at    1  -0.95
\put {$\bullet$}                 [t]  at    0  4

\put {$z_{1}$}                 [t]  at    0  0.8
\put {$z_{2}$}                 [t]  at    2  0.8
\put {$\w$}                 [b]  at    0.8  1.3
\put{$\cS_{z_{1}}(\w)$}  [b]  at    -0.2  1.7
\put{$\cS_{z_2}(\w)$}  [b]  at    1.8  1.45

\put {$x$}                 [t]  at    1.5  -1
\put {$y$}                 [t]  at    0.8  4

\arrow <10pt> [.2,.6] from 0.5 1.1 to 0.5 1.4
\setlinear
\plot    1 -1 -2  2  0 4   3 1 1 -1 /
\plot 0 1 0.5 1.8 /
\plot 0.3 1.7  2 1 /
\plot 0 1 -1 1.6 /
\plot 2 1 2.5 1.9 /
\color{red}
\plot 0 1 2 1 /

\endpicture}
\hfill{}
\color{black}
\caption{A type A folding when $CH(x,y)$ is a parallelogram. The retraction is centered at $\omega$. Here $CH(x,y)\cap \cS_{x}(\omega)\cap \cS_{y}(\omega)=[z_{1},z_{2}]\neq \varnothing.$}
\label{hyper}
\end{figure}

\subsection{The boundary partition}\label{Sec3}
For $\mathbf{n}\in \mathbb{N}^{2}$, and $y\in S_{\bf{n}}(x)$ and for $(\mathbf i,\mathbf j)\in\mathcal F(\mathbf n)$, define
\[
\mathcal F_{\mathbf i,\mathbf j}(x,y)
=\{\omega:\mathcal F(x,y,\omega)=(\mathbf i,\mathbf j)\}.
\]
These sets give the disjoint partition
\begin{equation}\label{partitionF}
\Omega=\coprod_{(\mathbf i,\mathbf j)\in\mathcal F(\mathbf n)}
\mathcal F_{\mathbf i,\mathbf j}(x,y).
\end{equation}
Reversal for type B foldings and equivariance give
\begin{align*}
\mathcal F_{\mathbf i,\mathbf j}(x,y)
&=\mathcal F_{\mathbf j,\mathbf i}(y,x),\\
\gamma\mathcal F_{\mathbf i,\mathbf j}(x,y)
&=\mathcal F_{\mathbf i,\mathbf j}(\gamma x,\gamma y).
\end{align*}

For a type~B folding with folding vertices $z_1,z_2$, we use
\[
\mathcal F_{\mathbf i,\mathbf j}(x,y)
\subseteq\Omega_x(z_1)\cap\Omega_y(z_2)
\cap\mathcal F_{0,0}(z_1,z_2).
\]

We call $\mathcal{F}_{0,0}(x,y)$ \emph{the extreme folding} of $CH(x,y).$
\subsection{Busemann functions and foldings}

We use the Radon--Nikodym formula~\eqref{RN}.
Only the scalar cocycle
$\langle\beta_\omega(x,y),\mathbf h\rangle$ will be needed below.

\begin{lemma}\label{busecomp}
Let $y\in S_{k,0}(x)\cup S_{0,k}(x)$. For every
$\omega\in\mathcal F_{0,0}(x,y)$,
\[
\langle\beta_\omega(x,y),\mathbf h\rangle=0.
\]
\end{lemma}

\begin{proof}
For $k=0$ the assertion is immediate. Otherwise, the extreme folding
determines a vertex $z\in\mathcal S_x(\omega)\cap\mathcal S_y(\omega)$
such that $x,y,z$ form an equilateral triangle of side length $k$.
The vector distances $\sigma(x,z)$ and $\sigma(y,z)$ are $(k,0)$
and $(0,k)$, in one order or the other. Thus
\[
\beta_\omega(x,y)=\sigma(x,z)-\sigma(y,z)
\in\{(k,-k),(-k,k)\},
\]
whose scalar product with $\mathbf h$ is zero.
\end{proof}

\begin{lemma}\label{busecomp1}\label{busecomp2}
If $\sigma(x,y)=\mathbf r$, then
\[
\beta_\omega(x,y)=\mathbf r
\quad\text{for }\omega\in\Omega_x(y),
\qquad
\beta_\omega(x,y)=-\sigma(y,x)
\quad\text{for }\omega\in\Omega_y(x).
\]
\end{lemma}

\begin{proof}
In the first case, $y$ belongs to the sector based at $x$ representing
$\omega$, so the defining formula for the Busemann cocycle gives
$\beta_\omega(x,y)=\sigma(x,y)$. For the second case, use the same
argument with the base point $y$ and the identity
$\beta_\omega(x,y)=-\beta_\omega(y,x)$.
\end{proof}

\begin{prop}\label{buscomp3}
Let $\mathcal F_{\mathbf i,\mathbf j}(x,y)$ be a type~B folding,
with folding vertices $a,b'$ satisfying
$\sigma(x,a)=\mathbf i$ and $\sigma(y,b')=\mathbf j$.
Then, for every $\omega\in\mathcal F_{\mathbf i,\mathbf j}(x,y)$,
\begin{equation}\label{eq:scalar-busemann}
\langle\beta_\omega(x,y),\mathbf h\rangle
=2\bigl(|\mathbf i|-|\mathbf j|\bigr).
\end{equation}
In particular,
\begin{equation}\label{eq:folding-rn}
\frac{d\nu_y}{d\nu_x}(\omega)
=q^{2(|\mathbf i|-|\mathbf j|)}.
\end{equation}
\end{prop}

\begin{proof}
For a boundary point in this folding, we have
\[
\omega\in\Omega_x(a)\cap\Omega_y(b')
\cap\mathcal F_{0,0}(a,b').
\]
The cocycle identity gives
\[
\beta_\omega(x,y)
=\beta_\omega(x,a)+\beta_\omega(a,b')+\beta_\omega(b',y).
\]
The first and last terms are $\mathbf i$ and $-\mathbf j$,
respectively. The middle term has scalar product zero with
$\mathbf h$ by Lemma~\ref{busecomp}. This proves
\eqref{eq:scalar-busemann}, and \eqref{eq:folding-rn} follows.
\end{proof}

\begin{remark}
Equation~\eqref{eq:scalar-busemann} is a scalar identity. The middle
vector $\beta_\omega(a,b')$ need not vanish, so one cannot in general
replace it by the vector identity
$\beta_\omega(x,y)=\mathbf i-\mathbf j$.
\end{remark}

\section{Boundary operators and foldings}\label{Sec4}

Write $H=L^2(\Omega,\nu_x)$ and $\pi=\pi_{\nu_x}$ for the boundary
representation introduced above.
For a measurable subset $E\subset\Omega$, denote multiplication by
$\mathbf1_E$ on $L^{2}(\Omega,\nu_{x})$ by the projection $p_E$. The representation formula gives
\begin{equation}\label{covariant}
\pi(\gamma)p_E\pi(\gamma)^{-1}=p_{\gamma E}.
\end{equation}

Fix $\mathbf n=(m,n)$.
For a folding index $(\mathbf i,\mathbf j)\in\mathcal F(\mathbf n)$,
define, for $\gamma\in S^\Gamma_{\mathbf n}(x)$,
\begin{equation}\label{defpiij}
\pi_{\mathbf i,\mathbf j}(\gamma)
=p_{\mathcal F_{\mathbf i,\mathbf j}(x,\gamma x)}\pi(\gamma).
\end{equation}
The partition into foldings gives
\begin{equation}\label{partitionP}
\pi(\gamma)
=\sum_{(\mathbf i,\mathbf j)\in\mathcal F(\mathbf n)}
\pi_{\mathbf i,\mathbf j}(\gamma).
\end{equation}
Moreover, equivariance and reversal for type B foldings imply
\begin{equation}\label{defpiijadjoint}
\pi_{\mathbf i,\mathbf j}(\gamma)^*
=\pi(\gamma^{-1})p_{\mathcal F_{\mathbf i,\mathbf j}(x,\gamma x)}
=\pi_{\mathbf j,\mathbf i}(\gamma^{-1}).
\end{equation}
Here the operator on the right is defined on the reverse sphere.

For $f$ supported on $S^\Gamma_{\mathbf n}(x)$, set
\[
\pi_{\mathbf i,\mathbf j}(f)
=\sum_{\gamma\in S^\Gamma_{\mathbf n}(x)}
f(\gamma)\pi_{\mathbf i,\mathbf j}(\gamma).
\]
Since each summand preserves nonnegative functions,
\begin{equation}\label{positivity}
\bigl|\langle\pi_{\mathbf i,\mathbf j}(f)\phi,\psi\rangle\bigr|
\le
\langle\pi_{\mathbf i,\mathbf j}(|f|)|\phi|,|\psi|\rangle.
\end{equation}

\section{Multiplicity and interpolation on the product boundary}

Fix $\mathbf n=(m,n)$, a type~B folding index
$(\mathbf i,\mathbf j)$, and vertices
$a\in S_{\mathbf i}(x)$, $b\in S_{\mathbf j}(x)$.
Set
\[
k=|\mathbf n|-|\mathbf i|-|\mathbf j|>0
\]
and
\[
S^\Gamma_{\mathbf n}(a,b)
=\left\{\gamma\in S^\Gamma_{\mathbf n}(x):
\begin{array}{l}
CH(x,a)\subseteq CH(x,\gamma x),\\
CH(x,b)\subseteq CH(x,\gamma^{-1}x)
\end{array}\right\}.
\]
On each nonempty compressed folding piece, the folding vertices
$a$ and $\gamma b$ are joined by a singular segment of length $k$.
We choose the orientation with source direction $(0,1)$ and target
direction $(1,0)$; interchanging the two directions gives the same
argument.

Write
\[
p_a=p_{\Omega_x(a)},\qquad p_b=p_{\Omega_x(b)},\qquad
C_\gamma=p_a\pi_{\mathbf i,\mathbf j}(\gamma)p_b.
\]
On the product boundary
\[
Z=\Omega\times\Omega,\qquad \mu=\nu_x\otimes\nu_x,
\]
with the diagonal action of $\Gamma$, consider
\begin{equation}\label{operatorT}
T=\sum_{\gamma\in S}C_\gamma\otimes C_\gamma.
\end{equation}
For a measurable subset $E\subseteq Z$, denote multiplication by
$\mathbf1_E$ by $p_E$. For $B\subseteq\Omega$, write $B^2=B\times B$.

\subsection{Branching partitions}

For $0\le i\le k$, $s\in S_{0,i}(b)$ and $t\in S_{i,0}(a)$, define
\begin{align*}
U_i(s)
&=\coprod_{\substack{u,u'\in S_{0,k}(b)\\
[b,u]\cap[b,u']=[b,s]}}
\Omega_b(u)\times\Omega_b(u'),\\
V_i(t)
&=\coprod_{\substack{v,v'\in S_{k,0}(a)\\
[a,v]\cap[a,v']=[a,t]}}
\Omega_a(v)\times\Omega_a(v').
\end{align*}
Here $i$ is the common initial length, and $s,t$ are the branching
vertices. Put
\[
U_i=\coprod_{s\in S_{0,i}(b)}U_i(s),\qquad
V_i=\coprod_{t\in S_{i,0}(a)}V_i(t).
\]
The shadow partitions give
\[
Z=\coprod_{i=0}^k U_i=\coprod_{j=0}^k V_j.
\]
Setting $T_{i,j}=p_{V_j}Tp_{U_i}$, we obtain
\begin{equation}\label{eq:full-block-decomposition}
T=\sum_{i,j=0}^k T_{i,j}.
\end{equation}

Suppose that $u,u'\in S_{0,k}(b)$ branch at $s\in S_{0,i}(b)$,
and put $\ell=k-i$.
The singular-branching calculation of
\cite[proof of Lemma~3.2]{RRS} gives
$\sigma(u,u')=(\ell,\ell)$.
For $\ell>0$, $CH(u,u')$ is a diamond of side length $\ell$,
with acute vertices $u,u'$ and obtuse vertices $s,s'$.
The opposite vertex $s'$ is unique \cite[Lemma~1.4]{RRS}; we write
\[
s'=\mathcal D_s^{0,1}(u,u'),\qquad \sigma(s,s')=(\ell,0).
\]
Similarly, for $v,v'\in S_{k,0}(a)$ branching at $t\in S_{i,0}(a)$,
\[
t'=\mathcal D_t^{1,0}(v,v'),\qquad \sigma(t,t')=(0,\ell).
\]
These maps are defined on the endpoint pairs just described.
When $\ell=0$, we set $s'=s$ and $t'=t$.

\begin{lemma}\label{lem:diamond-compatibility}
Let $\gamma\in S$, $0\le i,j\le k$,
$s\in S_{0,i}(b)$, $t\in S_{j,0}(a)$. Suppose $V_j(t)\cap\mathcal F_{0,0}(a,\gamma b)^2\cap\gamma U_i(s) \neq \varnothing$ and let $z=(\omega,\omega')\in
V_j(t)\cap\mathcal F_{0,0}(a,\gamma b)^2\cap\gamma U_i(s)$.
%\todo[inline]{YES ! Is it better to say: Suppose $V_j(t)\cap\mathcal F_{0,0}(a,\gamma b)^2\cap\gamma U_i(s) \neq \varnothing$ and let $z=(\omega,\omega')\in V_j(t)\cap\mathcal F_{0,0}(a,\gamma b)^2\cap\gamma U_i(s)$.}
And let $u,u'\in S_{0,k}(b)$, and $v, v'\in S_{k,0}(a)$ be such that $(\omega,\omega')\in \Omega_a(v)\times\Omega_a(v')$, $(\gamma^{-1}\omega,\gamma^{-1}\omega')\in \Omega_b(u)\times\Omega_b(u)$.
%\todo[inline]{YES ! Is it better to say: Let $u,u', v, v'$ be such that $(\omega,\omega')\in \Omega_a(v)\times\Omega_a(v')$, $(\gamma^{-1}\omega,\gamma^{-1}\omega')\in \Omega_b(u)\times\Omega_b(u)$?}

Then $v=\gamma u$, $v'=\gamma u'$ and $i=j$.
Moreover, for $s'=\mathcal D_s^{0,1}(u,u')$ and
$t'=\mathcal D_t^{1,0}(v,v')$,
\begin{equation}\label{eq:diamond-identifications}
\gamma s=t',\qquad \gamma s'=t.
\end{equation}
\end{lemma}

\begin{proof}
The extreme folding identifies the level vertices as the common
apices of the equilateral triangles of Lemma~\ref{busecomp}.
Thus $v=\gamma u$ and $v'=\gamma u'$.
Since $\gamma$ preserves vector distance,
\[
(k-i,k-i)=\sigma(u,u')=\sigma(v,v')=(k-j,k-j),
\]
so $i=j$.

For $\ell=k-i>0$, the intermediate distances
\[
\sigma(v,\gamma s)=(\ell,0),\qquad
\sigma(\gamma s,v')=(0,\ell)
\]
uniquely characterize $t'$. Hence $\gamma s=t'$.
The other obtuse vertex of $\gamma CH(u,u')=CH(v,v')$
is therefore $\gamma s'=t$.
When $\ell=0$, these identities follow directly from
$u=u'=s=s'$ and $v=v'=t=t'$.
See Figure~\ref{trapez}.
\end{proof}

In particular, Lemma~\ref{lem:diamond-compatibility} implies that for every $\gamma\in S$, $V_j\cap\mathcal F_{0,0}(a,\gamma b)^2\cap\gamma U_i
=\varnothing$ whenever $i\neq j$.
% \[
% V_j\cap\mathcal F_{0,0}(a,\gamma b)^2\cap\gamma U_i
% =\varnothing
% \qquad\text{if }i\ne j.
% \]

The corresponding operator summands are supported on these
intersections, so $T_{i,j}=0$ for $i\ne j$. Consequently,
\begin{equation}\label{eq:diagonal-block-decomposition}
T=\sum_{i=0}^k T_i,\qquad T_i=p_{V_i}Tp_{U_i}.
\end{equation}
For $s\in S_{0,i}(b)$ and $t\in S_{i,0}(a)$, set
\[
T_i(s,t)=p_{V_i(t)}Tp_{U_i(s)}.
\]
The branching partitions then give
\[
T_i=\sum_{s,t}T_i(s,t).
\]

\refstepcounter{picture}
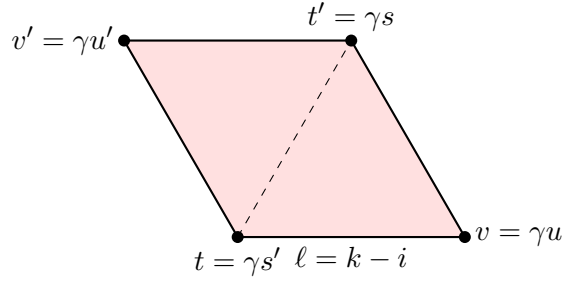
\begin{figure}[htbp]
\centering
\begin{tikzpicture}[scale=1.5]
\coordinate (T) at (0,0);
\coordinate (V) at (2,0);
\coordinate (Tp) at (1,1.73205);
\coordinate (Vp) at (-1,1.73205);

\fill[red!12] (T)--(V)--(Tp)--(Vp)--cycle;
\draw[thick] (T)--(V)--(Tp)--(Vp)--cycle;
\draw[dashed] (T)--(Tp);

\foreach \p in {T,V,Tp,Vp}
  \fill (\p) circle (1.5pt);

\node[below] at (T) {$t=\gamma s'$};
\node[right] at (V) {$v=\gamma u$};
\node[above] at (Tp) {$t'=\gamma s$};
\node[left] at (Vp) {$v'=\gamma u'$};
\node[below] at (1,0) {$\ell=k-i$};
\end{tikzpicture}
\caption{The common diamond $CH(v,v')$, with opposite obtuse
vertices $t=\gamma s'$ and $t'=\gamma s$.}
\label{trapez}
\end{figure}
\subsection{Multiplicity}

Norms on $\Omega$ and $Z$ are taken with respect to $\nu_x$ and
$\mu=\nu_x\otimes\nu_x$, respectively.
For $\gamma\in S$, set
\[
A_i(\gamma,s,t)
=V_i(t)\cap\mathcal F_{0,0}(a,\gamma b)^2\cap\gamma U_i(s).
\]

\begin{defi}[Multiplicity]
The multiplicity of an indexed family $(E_\gamma)_{\gamma\in S}$
of subsets of $Z$ is
\[
\sup_{z\in Z}\sum_{\gamma\in S}\mathbf1_{E_\gamma}(z).
\]
\end{defi}

\begin{prop}\label{prop:uniform-multiplicity}
For every $i,s,t$, the family
$(A_i(\gamma,s,t))_{\gamma\in S}$ has multiplicity at most
$K_\Gamma$.
\end{prop}

\begin{proof}
Fix $z\in Z$. If no element contributes, the claim is immediate.
Otherwise $z\in V_i(t)$ determines $v,v'$ and
$t'=\mathcal D_t^{1,0}(v,v')$, independently of $\gamma$.
By Lemma~\ref{lem:diamond-compatibility}, every contributing element
satisfies $\gamma s=t'$. Hence
\[
\#\{\gamma\in S:z\in A_i(\gamma,s,t)\}
\le\#\{\gamma\in\Gamma:\gamma s=t'\}
=|\Gamma_s|\le K_\Gamma.
\]
\end{proof}

For fixed $i,s,t$, the $\gamma$-summand of $T_i(s,t)$ has output
support
\begin{equation}\label{eq:exact-support}
E_\gamma
=V_i(t)\cap\Omega_x(a)^2
\cap\mathcal F_{\mathbf i,\mathbf j}(x,\gamma x)^2
\cap\gamma\Omega_x(b)^2\cap\gamma U_i(s).
\end{equation}
The folding is extreme between $a$ and $\gamma b$ on this compressed
piece, so $E_\gamma\subseteq A_i(\gamma,s,t)$. The proposition gives
\begin{equation}\label{eq:direct-multiplicity}
\sum_{\gamma\in S}\mathbf1_{E_\gamma}(z)\le K_\Gamma.
\end{equation}

For the inverse family, fix $z\in Z$; again the bound is trivial
if no element contributes. Otherwise $z\in U_i(s)$ determines the
source level vertices $u,u'$ and $s'=\mathcal D_s^{0,1}(u,u')$.
If $z\in\gamma^{-1}E_\gamma$, applying
Lemma~\ref{lem:diamond-compatibility} to $\gamma z$ gives
$\gamma s'=t$. Thus
\begin{equation}\label{eq:inverse-multiplicity}
\sum_{\gamma\in S}\mathbf1_{\gamma^{-1}E_\gamma}(z)
\le\#\{\gamma\in\Gamma:\gamma s'=t\}
=|\Gamma_{s'}|\le K_\Gamma.
\end{equation}

\subsection{The adjoint and Riesz--Thorin interpolation}

\begin{theorem}\label{theoRT}
For $0\le i\le k$, $s\in S_{0,i}(b)$ and $t\in S_{i,0}(a)$,
\[
\|T_i(s,t)\|_{L^2(Z,\mu)\to L^2(Z,\mu)}\le K_\Gamma.
\]
\end{theorem}

\begin{proof}
Put $d=|\mathbf i|-|\mathbf j|$.
By \eqref{eq:folding-rn}, the two tensor factors give weight $q^{2d}$.
Using also \eqref{defpiijadjoint}, we obtain
\begin{align}
(T_i(s,t)G)(z)
&=q^{2d}\sum_{\gamma\in S}\mathbf1_{E_\gamma}(z)
G(\gamma^{-1}z),\label{eq:block-action}\\
(T_i(s,t)^*G)(z)
&=q^{-2d}\sum_{\gamma\in S}\mathbf1_{\gamma^{-1}E_\gamma}(z)
G(\gamma z).\label{eq:reversed-block-adjoint}
\end{align}
The multiplicity bounds and duality therefore give
\begin{align}
\|T_i(s,t)\|_{\infty\to\infty}
&\le K_\Gamma q^{2d},\label{eq:block-infty}\\
\|T_i(s,t)\|_{1\to1}
=\|T_i(s,t)^*\|_{\infty\to\infty}
&\le K_\Gamma q^{-2d}.\label{eq:block-adjoint-infty}
\end{align}
Riesz--Thorin interpolation at parameter $1/2$ yields
\[
\|T_i(s,t)\|_{2\to2}
\le\|T_i(s,t)\|_{1\to1}^{1/2}
   \|T_i(s,t)\|_{\infty\to\infty}^{1/2}
\le K_\Gamma.
\]

\end{proof}

\subsection{Summing the diagonal blocks}

\begin{lemma}\label{lem:localized-tensors}
For $\phi,\psi\in H$,
\[
\|p_{U_i(s)}(\phi\otimes\psi)\|_2
\le\|\mathbf1_{\Omega_b(s)}\phi\|_2
   \|\mathbf1_{\Omega_b(s)}\psi\|_2.
\]
The analogous estimate holds for $V_i(t)$ with $a,t$ in place of $b,s$.
\end{lemma}

\begin{proof}
The inclusions $U_i(s)\subseteq\Omega_b(s)^2$ and
$V_i(t)\subseteq\Omega_a(t)^2$ give both estimates by Fubini.
\end{proof}

\begin{theorem}\label{prop:diagonal-sum}
For $\phi,\psi\in H$, set
\[
\mathcal Q_k(\phi,\psi)
=\langle T(\phi\otimes\phi),\psi\otimes\psi\rangle.
\]
Then
\begin{equation}\label{eq:final-form-bound}
|\mathcal Q_k(\phi,\psi)|
\le K_\Gamma(k+1)\|\phi\|_2^2\|\psi\|_2^2.
\end{equation}
\end{theorem}

\begin{proof}
For each $i$, Theorem~\ref{theoRT} and
Lemma~\ref{lem:localized-tensors} give
\begin{align*}
|\langle T_i(\phi\otimes\phi),\psi\otimes\psi\rangle|& \le \sum_{s,t} | \langle T_{i}(s,t)(\phi\otimes\phi),\psi \otimes \psi \rangle |\\
&\le K_\Gamma\sum_{s,t}
\|p_{U_i(s)}(\phi\otimes\phi)\|_2
\|p_{V_i(t)}(\psi\otimes\psi)\|_2\\
&\le K_\Gamma
\left(\sum_s\|\mathbf1_{\Omega_b(s)}\phi\|_2^2\right)
\left(\sum_t\|\mathbf1_{\Omega_a(t)}\psi\|_2^2\right)\\
&=K_\Gamma\|\phi\|_2^2\|\psi\|_2^2.
\end{align*}
The last equality uses the shadow partitions at the level $i$.
Summing over $i=0,\ldots,k$ and using the equality (\ref{eq:diagonal-block-decomposition}) prove the claim. 
% \todo[inline]{I don't understand this proof. Isn't Theorem \ref{theoRT} already immediately implying this theorem? Why would Lemma \ref{lem:localized-tensors} have anything to do here?}
% \todo[inline]{Well, I do not know ! How ?? You sum over all $s,t$ then ! Lemma 5.4 is really important this where we you lose the $i$-information. But impossible to keep on tracking the dependance on $i$...}
% \todo[inline]{NO ! Is it true: $|\langle T_i(\phi\otimes\phi),\psi\otimes\psi\rangle|\le \|T_i\| \|\phi\otimes\phi\|\|\psi\otimes\psi\| \le K_\Gamma \|\phi\otimes\phi\|\|\psi\otimes\psi\|$?}
% \todo[inline]{EXACT ! okay, now I see that 5.4 only gives estimate the norm of $T_i(s,t)$, not $T_i$}
% \todo[inline]{This is where I satrted to try to find the right polynomial...IT is very difficult I think...}
\end{proof}

\subsection{The link with matrix coefficients}

For nonnegative $\phi,\psi\in H$, preservation of nonnegative
functions by each $C_\gamma$ gives
\begin{equation}\label{eq:tensor-coefficients}
\langle T(\phi\otimes\phi),\psi\otimes\psi\rangle
=\sum_{\gamma\in S}|\langle C_\gamma\phi,\psi\rangle|^2.
\end{equation}
Together with \eqref{eq:final-form-bound}, this identity yields
the following estimate.

\begin{coro}\label{cor:coefficient-bound}
For every function $f$ supported on $S$,
\[
\left\|\sum_{\gamma\in S}f(\gamma)C_\gamma\right\|
\le\sqrt{K_\Gamma(k+1)}\,\|f\|_{\ell^2(S)}.
\]
\end{coro}

\begin{proof}
For arbitrary $\phi,\psi\in H$, preservation of nonnegative
functions and Cauchy--Schwarz give
\begin{align*}
\left|\left\langle\sum_{\gamma\in S}f(\gamma)C_\gamma\phi,
\psi\right\rangle\right|
&\le\|f\|_2
\left(\sum_{\gamma\in S}
|\langle C_\gamma|\phi|,|\psi|\rangle|^2\right)^{1/2}\\
&\le\sqrt{K_\Gamma(k+1)}\|f\|_2\|\phi\|_2\|\psi\|_2.
\end{align*}
Taking the supremum over unit vectors completes the proof.
\end{proof}

\section{Type A foldings and assembly of the estimates}

\subsection{A multiplicity estimate}

The type~A estimate uses the fact that the first folding vertex
belongs to both representing sectors; see Figure~\ref{hyper}.

\begin{lemma}\label{lem:rectangular-multiplicity}
Let $F\subset\Gamma$ be finite, and let $E_\gamma\subset\Omega$
be measurable sets such that
\[
\sum_{\gamma\in F}
\mathbf1_{E_\gamma}(\xi)\mathbf1_{\gamma^{-1}E_\gamma}(\eta)
\le K
\qquad\text{for almost every }(\xi,\eta)\in\Omega^2.
\]
Then, for every $f:F\to\mathbb C$,
\[
\left\|\sum_{\gamma\in F}
f(\gamma)p_{E_\gamma}\pi(\gamma)\right\|
\le\sqrt K\,\|f\|_{\ell^2(F)}.
\]
\end{lemma}

\begin{proof}
Put $T_\gamma=p_{E_\gamma}\pi(\gamma)$.
Cauchy--Schwarz on $E_\gamma$, followed by a change of variables
using the Radon--Nikodym factor in $\pi(\gamma)$, gives
\[
|\langle T_\gamma\phi,\psi\rangle|^2
\le
\left(\int_{\gamma^{-1}E_\gamma}|\phi|^2\,d\nu_x\right)
\left(\int_{E_\gamma}|\psi|^2\,d\nu_x\right).
\]
Summing over $\gamma$ and applying the multiplicity bound yields
\[
\sum_{\gamma\in F}|\langle T_\gamma\phi,\psi\rangle|^2
\le K\|\phi\|_2^2\|\psi\|_2^2.
\]
Cauchy--Schwarz in $\gamma$ now proves the operator estimate.
\end{proof}

\subsection{Grouping the type A foldigns}

Fix $\mathbf n=(m,n)$ and let $f$ be supported on
$S^\Gamma_{\mathbf n}(x)$.
For each possible first folding coordinate
$\mathbf i=(i_1,i_2)$, set
\[
E^A_{\mathbf i}(\gamma)
=\coprod_{\substack{\mathbf j\\
(\mathbf i,\mathbf j)\in\mathcal F_A(\mathbf n)}}
\mathcal F_{\mathbf i,\mathbf j}(x,\gamma x),
\qquad
A_{\mathbf i}(f)
=\sum_{\gamma\in S^\Gamma_{\mathbf n}(x)}
f(\gamma)p_{E^A_{\mathbf i}(\gamma)}\pi(\gamma).
\]
For a type~A folding, the first folding vertex satisfies
\[
z_1\in CH(x,\gamma x)\cap\mathcal S_x(\omega)
\cap\mathcal S_{\gamma x}(\omega).
\]

\begin{prop}\label{prop:type-a}
For every possible $\mathbf i$,
\[
\|A_{\mathbf i}(f)\|\le\sqrt{K_\Gamma}\,\|f\|_2.
\]
Consequently,
\[
\left\|\sum_{\mathbf i}A_{\mathbf i}(f)\right\|
\le\sqrt{K_\Gamma}(m+1)(n+1)\|f\|_2.
\]
\end{prop}

\begin{proof}
Fix $(\xi,\eta)\in\Omega^2$ and suppose that
\[
\xi\in E^A_{\mathbf i}(\gamma),
\qquad
\eta\in\gamma^{-1}E^A_{\mathbf i}(\gamma).
\]
Let $z$ be the unique vertex of $\mathcal S_x(\xi)$ with
$\sigma(x,z)=\mathbf i$.
It is the first folding vertex for $\xi$.
The first folding vertex for $\gamma\eta$ is also $z$:
both belong to $CH(x,\gamma x)$ and have coordinate $\mathbf i$
based at $x$.
The type~A property therefore gives
\[
z\in\mathcal S_{\gamma x}(\gamma\eta),
\qquad
z'=\gamma^{-1}z\in\mathcal S_x(\eta).
\]
Coordinates in $CH(x,\gamma x)$ give
\[
\sigma(x,z')=\sigma(\gamma x,z)=(n-i_2,m-i_1).
\]
Thus $z$ and $z'$ are determined by
$\xi,\eta,\mathbf n,\mathbf i$, independently of $\gamma$.
Every contributing element satisfies $\gamma z'=z$, and
\[
\#\{\gamma\in\Gamma:\gamma z'=z\}\le K_\Gamma.
\]
Lemma~\ref{lem:rectangular-multiplicity} proves the first estimate.
Since $0\le i_1\le m$ and $0\le i_2\le n$, there are at most
$(m+1)(n+1)$ possible first coordinates.
The triangle inequality gives the second estimate.
The argument also covers degenerate configurations.
\end{proof}

\subsection{Combining the two types}

Fix a type~B folding $(\mathbf i,\mathbf j)$, with residual
singular length $k\le m+n$.
For $a\in S_{\mathbf i}(x)$ and $b\in S_{\mathbf j}(x)$, put
\[
f_{a,b}=f\,\mathbf1_{S^\Gamma_{\mathbf n}(a,b)},
\qquad
B_{a,b}(f)=p_a\pi_{\mathbf i,\mathbf j}(f)p_b.
\]
Only elements of $S^\Gamma_{\mathbf n}(a,b)$ contribute to
this block, so Corollary~\ref{cor:coefficient-bound} gives
\[
\|B_{a,b}(f)\|
\le\sqrt{K_\Gamma(k+1)}\,\|f_{a,b}\|_2.
\]
For this fixed folding, each $\gamma$ determines a unique pair
$(a,b)$, hence
\[
\sum_{a,b}\|f_{a,b}\|_2^2=\|f\|_2^2.
\]
The shadow projections $p_a$ and $p_b$ form orthogonal partitions.
The operator matrix estimate \cite[Lemma~1.2]{RRS} therefore gives
\[
\begin{aligned}
\|\pi_{\mathbf i,\mathbf j}(f)\|^2
&\le\sum_{a,b}\|B_{a,b}(f)\|^2\\
&\le K_\Gamma(k+1)\sum_{a,b}\|f_{a,b}\|_2^2\\
&=K_\Gamma(k+1)\|f\|_2^2.
\end{aligned}
\]

Recall the polynomial $P_{\mathrm{RRS}}(m,n)$ defined in the
introduction.
By \cite[Lemma~3.1]{RRS},
\[
|\mathcal F_B(m,n)|
=\tfrac12(m+1)(n+1)(m+n).
\]
Combining the type~B estimate above with
Proposition~\ref{prop:type-a}, and using $k+1\le m+n+1$, yields
\begin{equation}\label{eq:all-foldings-bound}
\begin{aligned}
\|\pi(f)\|
&\le\sqrt{K_\Gamma}
\bigl((m+1)(n+1)
+|\mathcal F_B(m,n)|\sqrt{m+n+1}\bigr)\|f\|_2\\
&\le\sqrt{K_\Gamma}\,P_{\mathrm{RRS}}(m,n)\sqrt{m+n+1}\,\|f\|_2.
\end{aligned}
\end{equation}

\section{Vector-radial functions and Riesz--Thorin interpolation}
\label{sec:radial-boundary}

Fix a nonempty $\Gamma$-sphere $S=S^\Gamma_{m,n}(x)$, with $m,n\in \mathbb{N}$ such that $m+n>0$,
and let $f=c\mathbf1_{S_{m,n}}$ be vector-radial on $S$.
For detailed discussions of the counting results presented in this section, we refer the reader to~\cite{C} and~\cite{CMS}. Write $N_{m,n}=|S_{m,n}(x)|$ 
for the cardinality of the spheres in $X$. Notice that it is independent of $x$ and is given by
\begin{align}\label{cgspheres}
N_{m,n} =
\begin{cases}
(q^2 +q+1)(q^2+q) q^{2(m+n-2)} & \text{if }m,n\geq 1 , \\
(q^2 +q+1)q^{2(m-1)} & \text{if }m\geq 1 , n=0 , \\
(q^2 +q+1)q^{2(n-1)} & \text{if }n\geq 1 , m=0 , \\
1 & \text{if }m=0=n.
\end{cases}
\end{align}and set  $$h_x=|\Gamma_x|.$$ In particular,
\begin{equation}\label{eq:radial-normalization}
\frac{q^2}{q^2+q+1}\sqrt{N_{m,n}}
\le q^{m+n}.
\end{equation}

Fix a type~B index $(\mathbf i,\mathbf j)$ and set
\[
 k=m+n-|\mathbf i|-|\mathbf j| \quad b_{q}=\frac{q^{2}}{1+q+q^{2}}.
\]

\subsection{From folding measures to multiplicities}

For a folding index $(\mathbf i,\mathbf j)$ and $y\in S_{m,n}(x)$, let $z_1,z_2$ be
its folding vertices, and put $k=d(z_1,z_2)$.

\begin{lemma}[Decomposition and measure of a folding]
\label{lem:radial-fold-measure}\label{decompofolding}
Assume $k>0$. Let $\mathcal T$ be the set of vertices $t$ such that
$\Delta(z_1,z_2,t)$ is an equilateral triangle of side $k$ and
\[
 t\notin CH(x,y),\qquad
 \Delta(z_1,z_2,t)\cap CH(x,y)=[z_1,z_2].
\]
Then
\begin{equation}\label{eq:radial-fold-cylinders}
 E=\coprod_{t\in\mathcal T}
       \bigl(\Omega_{z_1}(t)\cap\Omega_{z_2}(t)\bigr),
\end{equation}
and
\begin{equation}\label{eq:radial-fold-measures}
 \nu_{z_1}(\mathcal F_{\mathbf i,\mathbf j}(x,y))=\frac{q^2}{q^2+q+1}q^{-k}
 \begin{cases}
 (q-1)/(q+1),&\text{for type A},\\
 1,&\text{for type B}.
 \end{cases}
\end{equation}
\end{lemma}
\begin{proof}
The sector direction determines a unique equilateral triangle based
on $[z_1,z_2]$ and pointing towards the boundary point. Its apex
belongs to $\mathcal T$ precisely when the induced folding has
vertices $z_1,z_2$. Uniqueness of the apex gives the disjoint
union~\eqref{eq:radial-fold-cylinders}.

There are $q+1$ admissible initial directions in type B and $q-1$
in type A, where the two directions entering the convex hull are
excluded. Each has $q^{k-1}$ extensions. Thus
$|\mathcal T|=(q+1)q^{k-1}$ or $(q-1)q^{k-1}$, respectively.
For each $t$, choose the opposite apex $p$ in an apartment containing
$\Delta(z_1,z_2,t)$. The corresponding cylinder is $\Omega_p(t)$,
with $\sigma(p,t)=(k,k)$. The shadow measure and change of base point
therefore give
\[
 \nu_{z_1}\bigl(\Omega_{z_1}(t)\cap\Omega_{z_2}(t)\bigr)
 =\frac{q^{2k}}{N_{k,k}}
 =\frac{q^4q^{-2k}}{(q^2+q+1)(q^2+q)}.
\]
Multiplication by $|\mathcal T|$ proves
\eqref{eq:radial-fold-measures}.
\end{proof}

For the radial estimate, we use the type B consequence. 
On the folding set  $\mathcal{F}_{\mathbf i,\mathbf j}(x,y)$, $d\nu_x/d\nu_{z_1}=q^{-2|\bf{i}|}$. Hence Lemma \ref{decompofolding} gives,
in particular, the bound needed below:
\begin{equation}\label{eq:radial-fold-measure-bound}
 \nu_x\bigl(\mathcal F_{\mathbf i,\mathbf j}(x,y)\bigr)
 \le \frac{q^{|\mathbf j|-|\mathbf i|}}{\sqrt{N_{m,n}}}.
\end{equation}

We also need the counting version of this measure estimate. We spell
out the regularity input: the multiplicity below is constant as a
function of the boundary direction. This is not a consequence of
integrating an arbitrary family of measurable sets.

\subsection{Counting foldings in a fixed boundary direction}

The quantities relevant to interpolation are the counting function defined for $m,n\in \mathbb{N}$ and for $\bf{i,j}$ such that $\mathcal{F}_{\bf{i,j}}(\cdot,\cdot)$ is a folding:
\[
 F^{\bf{i,j}}_{m,n}(\omega)=q^{|\mathbf i|-|\mathbf j|}\sum_{\gamma\in S_{m,n}}\mathbf1_{\mathcal{F}_{\bf{i,j}}(x,\gamma x)}(\omega).
\]

\begin{lemma}\label{lem:radial-fold-count}
For every type~B folding of indexes $(\mathbf i,\mathbf j)$,
\begin{equation}\label{eq:radial-fold-count}
 \|F^{\mathbf i,\mathbf j}_{m,n}\|_\infty
 \le h_x b_q N_{m,n}q^{-(m+n)}
 \le h_x\sqrt{N_{m,n}}.
\end{equation}
\end{lemma}

\begin{proof}
Fix $\omega\in\Omega$, and set 
$ 
  k=(m+n)-|\mathbf{i}|-|\mathbf{j}|>0.
$ 
We first count vertices $y\in S_{m,n}(x)$ satisfying
$\omega\in\mathcal F_{\mathbf i,\mathbf j}(x,y)$.
The first folding vertex $a$ is fixed: it is the unique vertex of
$\mathcal S_x(\omega)$ with $\sigma(x,a)=\mathbf i$.
The orientation and length $k$ of the residual singular segment
are determined by the folding index. Hence the common apex $t$
of its extreme-folding triangle is also fixed in
$\mathcal S_a(\omega)$.

Write $b$ for the second folding vertex. Each admissible $y$
determines an equilateral triangle $\Delta(a,b,t)$ with the
singular side $[a,t]$ prescribed. The triangular-strip count
used in Lemma~\ref{decompofolding} gives at most
$(q+1)q^{k-1}$ possible triangles: there are $q+1$ initial
chamber choices along the prescribed side, and at most $q$
continuations at each of the remaining $k-1$ stages.
If $mn=0$, the convex hull $CH(x,y)$ is singular. The initial
continuation in the sector direction cannot realize the prescribed
first folding vertex $a$ and must be excluded. Thus in this case
there are at most $q^k$ possible triangles.

For each such triangle, we next bound the number of possible
endpoints $y$. They satisfy
\[
 \sigma(y,b)=\mathbf j,\qquad b\in\mathcal S_y(\omega).
\]
Choose a fixed order for the $|\bf{j}|$ singular unit steps from $y$
to $b$ inside $\mathcal S_y(\omega)$, and construct this path
backwards from $b$. At each step there are at most $q^2$
predecessors in the required direction. Indeed, in the projective
plane of order $q$ forming the vertex link, the sector direction
prescribes a point or a line; the admissible predecessors are
opposite to it, and their number is
$(q^2+q+1)-(q+1)=q^2$.
Consequently, each triangle gives at most $q^{2|\bf{j}|}$ endpoints $y$.
The additional conditions defining the folding and the sphere
can only reduce this number. Therefore
\[
 \#\{y\in S_{m,n}(x):
       \omega\in\mathcal F_{\mathbf i,\mathbf j}(x,y)\}
 \le q^{2|\bf{j}|}\cdot
 \begin{cases}
  (q+1)q^{k-1},&mn>0,\\
  q^k,&mn=0.
 \end{cases}
\]

Now let $S=S^\Gamma_{m,n}(x)$. Every vertex in $\Gamma x$
has exactly $h_x$ preimages under $\gamma\mapsto\gamma x$.
We obtain
\[
\begin{aligned}
 F^{\mathbf i,\mathbf j}_{m,n}(\omega)
 &\le h_x q^{|\bf{i}|-|\bf{j}|}q^{2|\bf{j}|} \cdot
 \begin{cases}
  (q+1)q^{k-1},&mn>0,\\
  q^k,&mn=0
 \end{cases}\\
 &=h_x \cdot
 \begin{cases}
  (q+1)q^{(m+n)-1},&mn>0,\\
  q^{m+n},&mn=0.
 \end{cases}
\end{aligned}
\]
By the sphere-cardinality formula~\eqref{cgspheres}, the last
expression is $h_x b_q N_{m,n}q^{-(m+n)}$. Finally,
\eqref{eq:radial-normalization} gives
$b_q N_{m,n}q^{-(m+n)}\le\sqrt{N_{m,n}}$.
The bound holds for every $\omega$, which proves the lemma.
\end{proof}

\subsection{Interpolation and the radial estimate}

\begin{prop}\label{prop:radial-type-b-RT}
For $T^{\bf{i,j}}_{m,n}=\pi_{\bf{i,j}}(\textbf{1}_{S_{m,n}})=\sum_{\gamma\in S_{m,n}}p_{\mathcal{F}_{\bf{i,j}}(x,\gamma x)}\pi(\gamma)$ one has
\begin{equation}\label{eq:radial-B-operator}
 \|T^{\bf{i,j}}_{m,n}\|_{2\to2}\le h_x\sqrt{N_{m,n}}.
\end{equation}
In particular, the bound is independent of $k$ and of the
chosen type~B index.
\end{prop}
\begin{proof}
Put $d=|\bf{i}|-|\bf{}j|$. The Radon--Nikodym formula
\eqref{eq:folding-rn} gives
\[
 (T^{\bf{i,j}}_{m,n}\phi)(\omega)
 =q^d\sum_{\gamma\in S_{m,n}}\mathbf1_{\mathcal{F}_{{\mathbf{i,j}(x,\gamma x)}}}(\omega)
                \phi(\gamma^{-1}\omega),
\]
so that $\|T^{\bf{i,j}}_{m,n}\|_{\infty\to\infty}\le \| F^{\bf{i,j}}_{m,n}\|_\infty$.
The operator need not be self-adjoint. By the properties of symmetry of folding,

\begin{align*}
(T^{\mathbf i,\mathbf j}_{m,n})^{*}
%&=q^{-d}\sum_{\gamma\in S_{m,n}}
%\mathbf 1_{\gamma^{-1}\mathcal F_{\mathbf i,\mathbf j}(x,\gamma x)}(\omega)
%\psi(\gamma\omega)\\
%&=q^{-d}\sum_{\gamma\in S_{m,n}}
%\mathbf 1_{\mathcal F_{\mathbf j,\mathbf i}%(x,\gamma^{-1}x)}(\omega)
%\psi(\gamma\omega)\\
%&=q^{-d}\sum_{\gamma\in S_{n,m}}
%\mathbf 1_{\mathcal F_{\mathbf j,\mathbf i}(x,\gamma x)}(\omega)
%\psi(\gamma^{-1}\omega)\\
&=T^{\mathbf j,\mathbf i}_{n,m}.
\end{align*}

whence $\|(T^{\bf{i,j}}_{m,n})^*\|_{\infty\to\infty}
\le \| F^{\bf{j,i}}_{n,m}\|_\infty$.
Using Lemma~\ref{lem:radial-fold-count},
Riesz--Thorin interpolation theorem now yields
\[
 \|T^{\bf{i,j}}_{m,n}\|_{2\to2}\le
 \bigl(\|T\|_{1\to1}\|T\|_{\infty\to\infty}\bigr)^{1/2}
 \le h_x\sqrt{N_{m,n}}.
\]
where the last equality follows from the sphere-cardinality
formula.
\end{proof}

\begin{coro}\label{cor:radial-RRS}
Assume that $\Gamma$ acts simply transitively on $X^0$.
For every type~B folding,
\[
 \|\pi_{\mathbf i,\mathbf j}(f)\|
 \le \rho_{m,n}(q)\|f\|_{\ell^2(\Gamma)}
 \le \|f\|_{\ell^2(\Gamma)},
\]
where
\[
 \rho_{m,n}(q)=
 \begin{cases}
 \displaystyle\sqrt{\frac{q(q+1)}{q^2+q+1}},&m,n>0,\\[2mm]
 \displaystyle\frac{q}{\sqrt{q^2+q+1}},&mn=0.
 \end{cases}
\]
Consequently,
\begin{equation}\label{eq:radial-final}
 \|\pi(f)\|\le
 \frac12(m+1)(n+1)(m+n+2)\,\|f\|_{\ell^2(\Gamma)}.
\end{equation}
\end{coro}
\begin{proof}
Simple transitivity action gives $|S_{m,n}|=N_{m,n}$, hence
$\|f\|_2=|c|\sqrt {N_{m,n}}$. By the sphere-cardinality formula,
$b_q\sqrt {N_{m,n}} q^{-(m+n)}=\rho_{m,n}(q)<1$, proving the first assertion.
The type~A estimate of Proposition~\ref{prop:type-a}
is $(m+1)(n+1)\|f\|_2$, while there are
$\frac12(m+1)(n+1)(m+n)$ type~B indices
\cite[Lemma~3.1]{RRS}. Summing gives~\eqref{eq:radial-final} that is Theorem \ref{theorem2}.
\end{proof}

\begin{remark}
The polynomial in~\eqref{eq:radial-final} is the one appearing in
the conjecture of Ramagge--Robertson--Steger~\cite{RRS}.
Here it is obtained for vector-radial functions for the boundary
representation, not for arbitrary functions. For a general proper
cocompact action, the same argument gives~\eqref{eq:radial-B-operator};
recovering a uniform multiple of $\|f\|_2$ also requires control
of the proportion of the sphere lying in $\Gamma x$.
\end{remark}

\end{document}